\documentclass[a4paper,12pt]{article}

\usepackage[left=2cm,right=2cm, top=2cm,bottom=3cm,bindingoffset=0cm]{geometry}

\usepackage{verbatim}
\usepackage{amsmath}
\usepackage{amsthm}
\usepackage{amssymb}
\usepackage{delarray}
\usepackage{cite}
\usepackage{hyperref}
\usepackage{mathrsfs}
\usepackage{tikz}
\usetikzlibrary{patterns}
\usepackage{caption}
\DeclareCaptionLabelSeparator{dot}{. }
\newcommand{\al}{\alpha}
\newcommand{\be}{\beta}
\newcommand{\ga}{\gamma}
\newcommand{\de}{\delta}
\newcommand{\la}{\lambda}
\newcommand{\om}{\omega}

\newcommand{\iy}{\infty}

\theoremstyle{plain}

\numberwithin{equation}{section}

\newtheorem{thm}{Theorem}[section]
\newtheorem{lem}[thm]{Lemma}

\theoremstyle{definition}

\newtheorem{example}[thm]{Example}

\newtheorem{ip}[thm]{Inverse Problem}
\newtheorem{df}[thm]{Definition}

\theoremstyle{remark}

\newtheorem{remark}[thm]{Remark}

\DeclareMathOperator{\diag}{diag}

 \allowdisplaybreaks

\begin{document}

\begin{center}
{\Large\bf Linear differential operators with distribution \\[0.2cm] coefficients of various singularity orders}
\\[0.2cm]
{\bf Natalia P. Bondarenko} \\[0.2cm]
\end{center}

\vspace{0.5cm}

{\bf Abstract.} In this paper, the linear differential expression of order $n \ge 2$ with distribution coefficients of various singularity orders is considered. We obtain the associated matrix for the regularization of this expression. Furthermore, we present the new statements of inverse spectral problems that consist in the recovery of differential operators with distribution coefficients from the Weyl matrix on the half-line and on a finite interval. The uniqueness theorems for these inverse problems are proved by developing the method of spectral mappings. 

\medskip


\medskip


\vspace{1cm}

\section{Introduction} \label{sec:intr}

This paper deals with the differential expression $\ell_n(y)$, $n \ge 2$,
defined as follows:
\begin{align} \nonumber
\ell_{2m+\tau}(y) := & y^{(2m + \tau)} + \sum_{k = 0}^{m-1} (-1)^{i_{2k} + k} (\sigma_{2k}^{(i_{2k})}(x) y^{(k)})^{(k)} \\ \label{defl} + & \sum_{k = 0}^{m + \tau - 2} (-1)^{i_{2k+1} + k + 1} \bigl[(\sigma_{2k+1}^{(i_{2k+1})}(x) y^{(k)})^{(k+1)} + (\sigma_{2k+1}^{(i_{2k+1})}(x) y^{(k+1)})^{(k)}\bigr], \: x \in \mathbb R_+,
\end{align}
where $m \in \mathbb N$, $\tau = 0, 1$, $n = 2m + \tau$, $\mathbb R_+ = (0, \infty)$; $(i_{\nu})_{\nu = 0}^{n-2}$ are integers such that $0 \le i_{2k + j} \le m - k - j$, $j = 0, 1$; $(\sigma_{\nu})_{\nu = 0}^{n-2}$ are regular functions, and the derivatives $\sigma_{\nu}^{(i_{\nu})}$ are understood in the sense of distributions.

Note that, if the functions $(\sigma_{\nu})_{\nu = 0}^{n-2}$ are sufficiently smooth, then the differential expression \eqref{defl} can be reduced to the form
\begin{equation} \label{horeg}
\ell_n(y) = y^{(n)} + \sum_{k = 0}^{n-2} p_k(x) y^{(k)}.
\end{equation}
However, if the coefficients are distributional, then it is convenient to consider differential expression in form \eqref{defl}. In particular, Mirzoev and Shkalikov \cite{MS16, MS19} have developed the regularization approach to the differential operators generated by $\ell_n(y)$ in the case of the maximal singularity orders $i_{2k + j} = m - k - j$, $j = 0, 1$


Linear differential operators $\ell_n(y)$ have a variety of physical applications, especially for $n = 2, 3, 4$. The second-order Sturm-Liouville (Schr\"odinger) operator $-\ell_2(y) = -y'' + q(x) y$ models string vibrations in classical mechanics, electron motion in quantum mechanics, and is also widely used in other branches of science. The third-order linear differential operators arise in the inverse problem method for integration of the nonlinear Boussinesq equation (see \cite{DTT82, McK81}), in mechanical problems of modeling thin membrane flow of viscous liquid and elastic beam vibrations (see \cite{BP19} and references therein). Spectral properties of the operator 
\begin{equation} \label{4th}
y^{(4)} + (p(x)y)' + q(x)y,
\end{equation}
which is the special case of $\ell_4(y)$, were studied in connection with the Euler-Bernoulli operator $\frac{1}{b(x)}(a(x)u'')''$ describing the beam vibrations (see, e.g., in \cite{BK17}). In \cite{Pol22}, the operator \eqref{4th} was considered in relation to the analysis of thin liquid polymer films of nanometer thickness.

The Schr\"odinger operators with distribution potentials are widely used in quantum mechanics for describing the interaction between individual particles \cite{Alb05}. Some aspects of spectral theory for the fourth-order differential operators with distribution coefficients were recently investigated in \cite{UB20, ZAB22}. The development of the general theory for higher-order differential operators with distribution coefficients could unify the approaches to specific problems in various applications, as well as causes interest from the purely mathematical point of view.

The goal of this paper is two-fold. First, we construct the regularization matrix for the differential expression \eqref{defl} with any derivative orders $(i_{\nu})_{\nu = 0}^{n-2}$. Second, we aim to study inverse spectral problems for differential operators generated by $\ell_n(y)$. Let us describe the background and the main results for each of these two issues in more detail. 

\subsection{Regularization}

To the best of the author's knowledge, ordinary differential operators with distribution coefficients have been investigated for more than 60 years.
The existence and uniqueness issues of initial value problem solution were considered in \cite{Kurz59, Pfaff79, Lev74, White79, Derr88} and other papers. A short overview of the early results in this direction can be found in \cite{KM19}. The most common way to treat differential equations with distribution coefficients is the reduction of such equations to first-order systems by introducing quasi-derivatives. In particular, for investigation of the second-order equation
$$
y'' + \sigma'(x) y = 0, \quad \sigma \in L_{2,loc}(a, b),
$$
Pfaff \cite{Pfaff79} transformed it to the system
$$
\begin{bmatrix} y \\ y^{[1]} \end{bmatrix}' = 
\begin{bmatrix}
-\sigma & 1 \\
-\sigma^2 & \sigma
\end{bmatrix}
\begin{bmatrix} y \\ y^{[1]} \end{bmatrix},
$$
where $y^{[1]} := y' + \sigma y$. In \cite{SS99}, Savchuk and Shkalikov started the systematic study of spectral theory for the Sturm-Liouville operators with distribution potentials, based on the same regularization. In the monograph~\cite{Weid87}, Weidmann studied the minimal and the maximal operators, deficiency indices, self-ajoint extensions, and some other issues of spectral theory for a class of higher-order matrix differential operators which includes operators generated by \eqref{defl} with $i_{2k} = 1$, $i_{2k+1} = 0$, in particular, the Sturm-Liouville operator with potential of $W_2^{-1}(a,b)$.

Later on, Mirzoev and Shkalikov \cite{MS16} have obtained the regularization matrix for the even-order differential expressions generalizing \eqref{defl} with $n = 2m$, $i_{2k+j} = m-k-j$, $j = 0, 1$. The analogous construction for the odd-order operators has been provided in \cite{MS19}. It is worth noting that, on a finite interval, the Mirzoev-Shkalikov case generalizes all the others with $0 \le i_{2k+j} \le m-k-j$, $j = 0, 1$. However, on the half-line, the classes of $\sigma \in L_1(\mathbb R_+)$ and $\sigma' \in L_1(\mathbb R_+)$ are not nested with one another. Therefore, the expression $\ell_n(y)$ is worth being studied for various $(i_\nu)_{\nu = 0}^{n-2}$.

Relying on the ideas of \cite{NZS99, Vlad04}, Vladimirov \cite{Vlad17} has developed an alternative approach for regularization of differential operators represented by bilinear forms. The construction of \cite{Vlad17} can be applied to a wider class of operators than \cite{MS16, MS19}. In particular, we use it in the present paper to obtain the regularization matrix for \eqref{defl}.  

In this paper, we assume that the coefficients at $y^{(n)}$ and $y^{(n-1)}$ equal $1$ and $0$, respectively. This assumption is natural for studying inverse spectral problems. However, in \cite{MS16, MS19, Vlad17}, the coefficients at $y^{(n)}$, $y^{(n-1)}$ can be arbitrary functions of certain classes. Also, note that we consider the cases of even and odd $n$ together, and the odd case appears to be easier for the purposes of this paper. 

Following the strategy of \cite{Vlad17}, we obtain the quadratic form and then construct the matrix $F(x) = [f_{k,j}(x)]_{k, j = 1}^n$ by a certain rule $F = \mathscr F_I(\Sigma)$, $I = (i_{\nu})_{\nu = 0}^{n-2}$, $\Sigma = (\sigma_{\nu})_{\nu = 0}^{n-2}$, associated with the differential expression \eqref{defl}. The matrix $F(x)$ has to fulfill the following conditions:
\begin{align*}
& f_{k, j}(x) \equiv 0, \quad k + 1 < j, \tag{i} \\
& f_{k, k + 1}(x) \equiv 1, \quad k = \overline{1, n-1}, \tag{ii}
\end{align*}
and $f_{k, j} \in L_1(\mathbb R_+)$, $0 \le j \le k \le n$.

\medskip

By using the matrix $F(x)$, define the quasi-derivatives:
\begin{equation} \label{quasi}
y^{[0]} := y, \quad y^{[k]} = (y^{[k-1]})' - \sum_{j = 1}^k f_{k,j} y^{[j-1]}, \quad k = \overline{1,n},
\end{equation}
and the domain
\begin{equation} \label{defDF}
\mathcal D_F = \{ y \colon y^{[k]} \in AC_{loc}(\mathbb R_+), \, k = \overline{0, n-1} \}.
\end{equation}

Our goal is to determine the rule $\mathscr F_I$ in such a way that $\ell_n(y) = y^{[n]}$ for any $y \in \mathcal D_F$. Then, instead of the equation $\ell_n(y) = \la y$, we can consider the equivalent system
\begin{equation} \label{sys}
{\vec y}\,' = (F(x) + \Lambda) \vec y, 
\end{equation}
where $\vec y(x) = \mbox{col} ( y^{[0]}(x), y^{[1]}(x), \ldots, y^{[n-1]}(x))$, $y \in \mathcal D_F$, $\la$ is the spectral parameter, $\Lambda := \la E_{n,1}$, $E_{n,1}$ is the $(n \times n)$ matrix whose element at $(n,1)$ equals $1$ and all the other elements equal $0$. Indeed, the first $(n-1)$ rows of the system~\eqref{sys} correspond to the quasi-derivative definition \eqref{quasi} and the $n$-th row is $y^{[n]} = \la y$. For the regular case $i_{\nu} = 0$, $\nu = \overline{0, n-2}$, such construction is well-known (see \cite{EM99}). The regularization in this paper generalizes the both regular and the Mirzoev-Shkalikov cases.

By using the special construction of the matrix $F(x)$ for the differential expression \eqref{defl}, we prove some important assertions (Lemmas~\ref{lem:PF-half}, \ref{lem:PF-finite}) for study of inverse problems.

\subsection{Inverse problems}

Inverse problems of spectral analysis consist in the recovery of operators from their spectral information. In terms of application, such problems correspond to determining unknown medium properties from some measured quantities or to constructing systems with desired characteristics.

The classical results of inverse spectral theory have been obtained for the Sturm-Liouville operator $-y'' + q(x) y$ with integrable potential $q(x)$ by Borg~\cite{Borg46}, Marchenko \cite{Mar77}, Levitan \cite{Lev84}, etc. The recovery of the Sturm-Liouville operators with singular potentials of classes $W_2^{\theta}$, $\theta \ge -1$, has been studied by Hryniv et al \cite{HM-sd, HM-2sp, HP12, HM20}, Freiling et al \cite{FIY08}, Savchuk and Shkalikov \cite{SS10}, Guliyev \cite{Gul19}. We also mention the studies of Mykytuyk and Trush \cite{MT09}, Eckhardt et al \cite{Eckh14, Eckh15}, Bondarenko \cite{Bond21-MN, Bond21-AMP} concerning inverse problems for the matrix Sturm-Liouville operator with distribution potentials.

Investigation of inverse problems for the higher-order differential operators \eqref{horeg} with $n > 2$ causes principal difficulties, since the classical transformation operator method (see \cite{Mar77, Lev84}) is ineffective for them. Relying on the ideas of Leibenson \cite{Leib66, Leib71}, Yurko has developed \textit{the method of spectral mappings}. This method has been used to construct the inverse problem theory for the higher-order differential operators \eqref{horeg} with regular coefficients on a finite interval and on the half-line (see \cite{Yur92, Yur93, Yur95-mn, Yur95-ms, Yur02}). Inverse problems on the line were studied by Beals et al \cite{Beals85, Beals88}. However, for the differential operators \eqref{defl} with distribution coefficients, there is still no general inverse problem theory. The first steps in this direction have been taken in \cite{Bond21}, where the uniqueness theorem has been proved for the inverse problem on a finite interval.

In this paper, we mostly focus on the inverse problem for the differential expression \eqref{defl} on the half-line. Let us provide the inverse problem statement.

Let $U = [u_{k,j}]_{k,j = 1}^n$ be a constant $(n \times n)$ matrix of form $U = P_U L_U$, where $P_U$ is a permutation matrix and $L_U$ is a unit lower-triangular matrix. This means that the matrix $P_U$ has the elements equal to $1$ at the positions $(k, p_k + 1)$, $k = \overline{1, n}$, where $\{ p_k \}_{k = 1}^n$ is the permutation of the numbers $\{ 0, 1, \ldots, n-1 \}$, and all the other elements are zero. The entries of $L_U = [l_{k,j}]_{k,j = 1}$ satisfy $l_{k,j} = \de_{k,j}$, $1 \le k \le j \le n$, where $\de_{k,j}$ is the Kronecker delta.

Consider the boundary value problem $\mathcal L_I(\Sigma, U)$ for the equation
\begin{equation} \label{eqv}
\ell_n(y) = \la y, \quad x \in \mathbb R_+,
\end{equation}
with the boundary conditions
\begin{equation} \label{bc}
U_s(y) := y^{[p_s]}(0) + \sum_{j = 1}^{p_s} u_{s,j} y^{[j-1]}(0) = 0, \quad s = \overline{1, n}.
\end{equation}

Denote by $\{ C_k(x,\la) \}_{k = 1}^n$ and $\{ \Phi_k(x, \la)\}_{k = 1}^n$ the solutions of equation~\eqref{eqv} satisfying the initial conditions
\begin{equation} \label{initC1}
U_s (C_k) = \de_{s,k}, \quad s = \overline{1, n},
\end{equation}
and the boundary conditions
\begin{equation} \label{bcPhi1}
    U_s(\Phi_k) = \de_{s,k}, \quad s = \overline{1, k}, \qquad
    \Phi_k(x, \la) = O(\exp(\rho \om_k x)), \quad x \to \infty,
\end{equation}
respectively. Here $\la = \rho^n$ and $\{ \om_k \}_{k = 1}^n$ are the roots of the equation $\om^n = 1$ numbered so that
$$
\mbox{Re}\, (\rho \om_1) \le \mbox{Re}\,(\rho \om_2) \le \ldots \le \mbox{Re}\,(\rho \om_m).
$$
The functions $\Phi_k(x, \la)$ are called \textit{the Weyl solutions} of \eqref{eqv}.

Consider the matrix functions $C(x, \la) = [\vec C_k(x, \la)]_{k = 1}^n$ and $\Phi(x, \la) = [\vec \Phi_k(x, \la)]_{k = 1}^n$. Since the columns of $C(x, \la)$ and $\Phi(x, \la)$ form two fundamental systems of solutions of \eqref{sys}, the relation $\Phi(x, \la) = C(x, \la) M(\la)$ is fulfilled, where $M(\la) = [M_{s,k}(\la)]_{s,k = 1}^n$ is called \textit{the Weyl matrix} and its entries $M_{s,k}(\la)$ are called \textit{the Weyl functions}. The conditions \eqref{initC1} and \eqref{bcPhi1} imply that $M(\la)$ is a unit lower-triangular matrix.

The Weyl functions and their generalizations are natural spectral characteristics in the inverse problem theory for various classes of differential operators and pencils (see, e.g., \cite{FY01}). The defined $M(\la)$ is analogous to the Weyl matrix used by Yurko \cite{Yur92, Yur93, Yur02} for solving inverse problems for the higher-order differential operators with regular coefficients. In this paper, we consider the following problem.

\begin{ip} \label{ip:main}
Suppose that $P_U$ is known a priori.
Given the Weyl matrix $M(\la)$, find $\Sigma$ and $U$.
\end{ip}

We study analytic properties of $M(\la)$ (Theorem~\ref{thm:Weyl}) and obtain the asymptotics of the Weyl solutions (Lemma~\ref{lem:asymptPhi}), by using the Birkhoff solutions of equation \eqref{eqv} with certain asymptotics as $|\rho| \to \infty$ (Theorems~\ref{thm:Birky1}, \ref{thm:Birkz}). By using the method of spectral mappings, we prove the uniqueness theorem for Inverse Problem~\ref{ip:main}. Furthermore, we consider the inverse problem on a finite interval previously studied in \cite{Bond21}. We discuss the usage of regularizations with various $(i_{\nu})_{\nu = 0}^{n-2}$ and the recovery of the boundary condition coefficients in a finite interval case. Note that our results are novel even in the regular case $i_{\nu} = 0$, $\nu = \overline{0, n-2}$. Since the method of spectral mappings is constructive, in the future, it can be used for developing an algorithm solving the studied inverse problems and for investigating existence of their solution.

In addition, we consider the examples for the orders $n = 2$ and $n = 4$, which often arise in applications. It is shown that the results of this paper generalize the previously known results for the Sturm-Liouville operators with regular and singular potentials. Also, we compare the inverse problems on the half-line and on a finite interval.

\medskip

The paper is organized as follows. Section~\ref{sec:reg} is devoted to the regularization issues. In Section~\ref{sec:birk}, theorems on Birkhoff systems are formulated. In Section~\ref{sec:ip}, we consider the inverse problem on the half-line, and in Section~\ref{sec:finite}, on a finite interval. Section~\ref{sec:ex} contains examples for $n = 2$ and $n = 4$.

\section{Regularization} \label{sec:reg}

In this section, we obtain the associated matrix $F(x)$ for regularization of the differential expression \eqref{defl} and study the properties of this matrix needed for the inverse problem theory.

Suppose that
\begin{align*} 
& I = (i_{\nu})_{\nu = 0}^{n-2} \in \mathcal I_n, \qquad
\mathcal I_{2m+\tau} := \left\{ (i_{\nu})_{\nu = 0}^{2m + \tau - 2} \colon \begin{array}{l} 
0 \le i_{2k} \le m-k, \, k = \overline{0, m-1}, \\
0 \le i_{2k+1} \le m - k - 1, \, k = \overline{0, m + \tau - 2}
\end{array}
\right\},
\\
& \Sigma = (\sigma_{\nu})_{\nu = 0}^{n-2} \in \Sigma_{I,loc}, \qquad
\Sigma_{I,loc} := \left\{ (\sigma_{\nu})_{\nu = 0}^{n-2} \colon
\begin{array}{l}
\sigma_{\nu} \in L_{1,loc}(\mathbb R_+), \quad \nu = \overline{0, n-2}, \\
\sigma_{\nu} \in L_{2,loc}(\mathbb R_+), \quad \nu \in K(I)
\end{array}
\right\},
\end{align*}
\vspace*{-20pt}
\begin{align*}
K(I) \subseteq \{ 0, 1, \ldots, n-2 \}, \qquad K(I) = \varnothing \quad & \text{if $n = 2m + 1$}, \\
\left.\begin{array}{rcl}
\nu = 2k \in K(I) & \Leftrightarrow & i_{2k} = m-k, \\
\nu = 2k+1 \in K(I) & \Leftrightarrow & i_{2k+1} = m-k-1
\end{array}\right\} \quad & \text{if $n = 2m$}.
\end{align*}

$\Sigma_{I}$ is defined similarly to $\Sigma_{I, loc}$ with $L_{j,loc}$ replaced by $L_j$, $j = 1, 2$. For the regularization matrix construction, it is sufficient to assume that $\Sigma \in \Sigma_{I, loc}$. The integrability of $\sigma_{\nu}(x)$ on $\mathbb R_+$ is important in the next sections for obtaining the Birkhoff solutions and for investigation of inverse problems.

Denote by $\mathfrak D$ the space $C_0^{\infty}(\mathbb R_+)$ of infinitely differentiable functions with a finite support on $\mathbb R_+$ and by $\mathfrak D'$ the space of all the continuous linear functionals on $\mathfrak D$.
For $z \in \mathfrak D$ and $f \in \mathfrak D'$, we use the notation $( f, z ) = fz$. In particular, $( f, z ) = \int_0^{\infty} f(x) z(x) \, dx$ if $f \in L_{1,loc}(\mathbb R_+)$.

\begin{lem} \label{lem:quad}
Suppose that $y \in W_{1,loc}^{m + \tau}(\mathbb R_+)$ if $K(I) = \varnothing$ and $y \in W_{2,loc}^m(\mathbb R_+)$ otherwise. Then $\ell_n(y) \in \mathfrak D'$ and 
\begin{equation} \label{quad}
( \ell_n(y), z ) = (-1)^m ( y^{(m + \tau)}, z^{(m)} ) + \sum_{r, j = 0}^m ( q_{r,j} y^{(r)}, z^{(j)} ), \quad z \in \mathfrak D,
\end{equation}
where
\begin{gather} \label{defQ}
[q_{r,j}]_{r,j = 0}^m = \mathscr Q_I(\Sigma) := \sum_{\nu = 0}^{n-2} \sigma_{\nu}(x) \chi_{\nu, i_{\nu}}, \quad \chi_{\nu, i} = [\chi_{\nu,i;r,j}]_{r,j = 0}^m, \\ \label{defchi}
\begin{array}{rl}
\chi_{2k, i; s + k, i - s + k} = & C_i^s, \quad s = \overline{0, i}, \\
\chi_{2k+1, i; s+k, i + 1- s + k} = & C_{i + 1}^s - 2 C_i^{s-1}, \quad s = \overline{0, i+1}, 
\end{array}
\end{gather}
and all the other entries $\chi_{\nu, i; r,j}$ equal zero. Here and below, $C_i^s = \frac{i!}{s!(i-s)!}$ are the binomial coefficients, $C_i^{-1} := 0$.
\end{lem}

Note that, at the right-hand side of \eqref{quad}, all the functions are regular, so \eqref{quad} describes an action of the functional $\ell_n(y) \in \mathfrak D'$ on an arbitrary $z \in \mathfrak D$. In fact, the relation~\eqref{quad} holds for $z$ of a wider class than $\mathfrak D$.

\begin{proof}[Proof of Lemma~\ref{lem:quad}]
Let $i, k \ge 0$. Formal calculations show that
\begin{align} \nonumber
\sigma^{(i)} y & = \sum_{s = 0}^i (-1)^s C_i^s (\sigma y^{(s)})^{(i - s)}, \\ \nonumber
(\sigma^{(i)} y^{(k)})^{(k)} & = \sum_{s = 0}^i (-1)^s C_i^s (\sigma y^{(s + k)})^{(i - s + k)}, \\ \label{quad-even}
( (\sigma^{(i)} y^{(k)})^{(k)}, z ) & = \sum_{s = 0}^i (-1)^{i + k} C_i^s ( \sigma y^{(s + k)}, z^{(i - s + k)}), \quad z \in \mathfrak D.
\end{align}

Clearly, under the conditions of the lemma, $\sigma_{2k} y^{(s + k)} \in L_{1,loc}(\mathbb R_+)$ for $s = \overline{0, i_{2k}}$, $k = \overline{0, n-2}$, so we conclude that $(\sigma_{2k}^{(i_{2k})} y^{(k)})^{(k)} \in \mathfrak D'$. 

Analogously,
\begin{align} \nonumber
    (\sigma^{(i)} y^{(k)})^{(k+1)} + (\sigma^{(i)} y^{(k+1)})^{(k)} & = \sum_{s = 0}^{i + 1} (-1)^s (C_{i+1}^s - 2 C_i^{s-1})(\sigma y^{(s + k)})^{(i + 1 - s + k)}, \\ \label{quad-odd}
    ( (\sigma^{(i)} y^{(k)})^{(k+1)} + (\sigma^{(i)} y^{(k+1)})^{(k)}, z) & = \sum_{s = 0}^{i + 1} (-1)^{i + k + 1} (C_{i + 1}^s - 2 C_i^s) ( \sigma y^{(s + k)}, z^{(i + 1 - s + k)}),
\end{align}
where $z \in \mathfrak D$. Taking $\sigma = \sigma_{2k+1}$, $i = i_{2k+1}$, we conclude that the terms with odd indices $(2k+1)$ in \eqref{defl} belong to $\mathfrak D'$ under the conditions of the lemma. Therefore, $\ell_n(y) \in \mathfrak D'$. Combining \eqref{defl}, \eqref{quad-even} with $\sigma = \sigma_{2k}$, $i = i_{2k}$, and \eqref{quad-odd} with $\sigma = \sigma_{2k+1}$, $i = i_{2k+1}$, we arrive at \eqref{quad}.
\end{proof}

Now, following the approach of Valdimirov \cite{Vlad17}, we are going to construct the matrix $F(x)$ of quasi-derivative coefficients by using the matrix $Q = [q_{l,s}]_{l,s = 0}^m$ of the quadratic form in \eqref{quad}. Define the spaces of matrix functions $\mathfrak Q_{n, loc}$ and $\mathfrak F_{n,loc}$ as follows:
\begin{gather*}
\mathfrak Q_{n, loc} = \left\{ Q = [q_{l,s}]_{l,s = 0}^m \colon q_{l,s} \in L_{1,loc}(\mathbb R_+); q_{l,m}, q_{m,s} \in L_{2,loc}(\mathbb R_+) \: \text{if $n = 2m$}, \: l,s = \overline{0, m} \right\}, \\
\mathfrak F_{n,loc} = \left\{ F = [f_{k,j}]_{k,j = 1}^n \colon \begin{array}{l}
\text{(i), (ii), (iii) are fulfilled}, \\ 
f_{k,j} \in L_{1,loc}(\mathbb R_+), \, k = \overline{m+1, 2m+\tau}, \, j = \overline{1, m+\tau}, \\
f_{k,m+1}, f_{m,j} \in L_{2,loc}(\mathbb R_+), \, k = \overline{m+1, 2m}, \, j = \overline{1, m}, \, \text{if $n = 2m$}
\end{array}
\right\}, \\
f_{k,j} = 0, \quad k = \overline{1, m - 1 + \tau},\, j \le k \:\: \text{and} \:\: j = \overline{m + 2, n},\, k \ge j. \tag{iii}
\end{gather*}
The spaces $\mathfrak Q_n$ and $\mathfrak F_n$ are defined similarly with $L_{1,loc}$ and $L_{2,loc}$ replaced by $L_1$ and $L_1 \cap L_2$, respectively. It follows from the definition of the mapping $\mathscr Q_I$ in Lemma~\ref{lem:quad} that 
\begin{equation} \label{mapQ}
\mathscr Q_I \colon \Sigma_{I,loc} \to \mathfrak Q_{n,loc} \quad \text{and} \quad \mathscr Q_I \colon \Sigma_{I} \to \mathfrak Q_{n}.
\end{equation}

Define the mapping $\mathscr S_n \colon \mathfrak Q_{n, loc} \to \mathfrak F_{n,loc}$ acting as follows:
\begin{align} \nonumber
    & F = \mathscr S_n(Q), \quad Q = [q_{l,s}]_{l,s = 0}^m \in \mathfrak Q_{n, loc}, \quad F = [f_{k,j}]_{k,j = 1}^n \in \mathfrak F_{n,loc}, \\ \label{defF}
    & n = 2m \colon \quad \begin{cases}
                            f_{m,j} := (-1)^{m+1} q_{j-1,m}, \quad j = \overline{1, m}, \\
                            f_{k,m+1} := (-1)^{k+1} q_{m,2m-k}, \quad k = \overline{m+1, 2m}, \\
                            f_{k,j} := (-1)^{k+1} q_{j-1,2m-k} + (-1)^{m+k} q_{j-1,m} q_{m,2m-k}, \quad k = \overline{m+1,2m}, \, j = \overline{1,m},
                        \end{cases} \\ \nonumber
    & n = 2m+1 \colon \quad f_{k,j} := (-1)^k q_{j-1, 2m+1-k},\, k = \overline{m+1, 2m+1}, \, j = \overline{1, m+1}.        
\end{align}
All the elements $f_{k,j}$ undefined here are uniquely specified by (i), (ii), and (iii). Obviously, 
\begin{equation} \label{mapS}
\mathscr S_n \colon \mathfrak Q_{n, loc} \to \mathfrak F_{n,loc} \quad \text{and} \quad \mathscr S_n \colon \mathfrak Q_{n} \to \mathfrak F_{n}.
\end{equation}

The inverse mapping $\mathscr S_n \colon \mathfrak Q_{n, loc} \to \mathfrak F_{n,loc}$ is given by the formulas:
\begin{align} \nonumber
    & Q = \mathscr S_n^{-1}(F), \quad F = [f_{k,j}]_{k,j = 1}^n \in \mathfrak F_{n,loc}, \quad Q = [q_{l,s}]_{l,s = 0}^m \in \mathfrak Q_{n, loc}, \\ \label{invS}
    & n = 2m \colon \quad \begin{cases}
                            q_{j-1,m} := (-1)^{m+1} f_{m,j}, \quad j = \overline{1, m}, \\
                            q_{m,2m-k} := (-1)^{k+1} f_{k,m+1}, \quad k = \overline{m+1, 2m}, \\
                            q_{j-1,2m-k} := (-1)^{k+1} (f_{k,j} - f_{k,m+1} f_{m,j}), \quad k = \overline{m+1,2m}, \, j = \overline{1,m}, \\
                            q_{m,m} := 0,
                        \end{cases} \\ \nonumber
    & n = 2m+1 \colon \quad q_{j-1, 2m+1-k} := (-1)^k f_{k,j},\, k = \overline{m+1, 2m+1}, \, j = \overline{1, m+1}.        
\end{align}

Define the mapping $\mathscr F_I(\Sigma) = \mathscr S_n(\mathscr Q_I(\Sigma))$. For a fixed $I \in \mathcal I_n$, the relations \eqref{mapQ} and \eqref{mapS} imply
\begin{equation} \label{mapF}
\mathscr F_I \colon \Sigma_{I, loc} \to \mathfrak F_{n,loc} \quad \text{and} \quad \mathscr F_I \colon \Sigma_I \to \mathfrak F_n.
\end{equation}

In fact, the above formulas defining the mapping $\mathscr S_n$ are the special case of the formulas on p.~6 of \cite{Vlad17}. In the Mirzoev-Shkalikov case $i_{2k+j} = m - k - j$, $j = 0, 1$, the matrix function $F(x) = \mathscr F_I(\Sigma)$ coincides with the associated matrices obtained in \cite{MS16, MS19}.

The following theorem establishes the equivalence of equation \eqref{eqv} and the system \eqref{sys} for $y \in \mathcal D_F$, $F(x) = \mathscr F_I(\Sigma)$.

\begin{thm} \label{thm:quasi}
Suppose that $\Sigma \in \Sigma_{I,loc}$ and $F = \mathscr F_I(\Sigma)$. Let the quasi-derivatives $y^{[j]}$, $j = \overline{0, n}$, and the domain $\mathcal D_F$ be defined by \eqref{quasi} and \eqref{defDF}, respectively.
Then, for any $y \in \mathcal D_F$, $\ell_n(y) \in L_{1,loc}(\mathbb R_+)$ and $\ell_n(y) = y^{[n]}$.
\end{thm}

\begin{proof}
For definiteness, consider $n = 2m$. The proof for $n = 2m+1$ is analogous and even easier. Since $F \in \mathfrak F_{n,loc}$, then the assumption (iii) holds, which together with \eqref{quasi} imply 
\begin{equation} \label{yj}
y^{[j]} = y^{(j)}, \quad j = \overline{0,m-1}. 
\end{equation}
Therefore, $y \in \mathcal D_F$ implies $y \in W_{1,loc}^m(\mathbb R_+)$. Moreover, 
\begin{equation} \label{ym}
y^{(m)} = y^{[m]} + \sum_{j = 1}^m f_{m,j} y^{(j-1)} \in L_{2,loc}(\mathbb R_+),
\end{equation}
so $y \in W_{2,loc}^m(\mathbb R_+)$. Thus, $y$ satisfies the conditions of Lemma~\ref{lem:quad}. Hence $\ell_n(y) \in \mathfrak D'$ and \eqref{quad} holds.

Using \eqref{quasi} and \eqref{yj}, we obtain
\begin{equation} \label{yk}
y^{[k]} = (y^{[k-1]})' - f_{k,m+1} y^{(m)} - \sum_{j = 1}^m (f_{k,j} - f_{k,m+1} f_{m,j}) y^{(j-1)}, \quad k = \overline{m+1, 2m}.
\end{equation}
Substituting \eqref{invS} into \eqref{yk}, we derive
$$
y^{[k]} = (y^{[k-1]})' + (-1)^k \sum_{j = 0}^m q_{j,2m-k} y^{(j)}, \quad k = \overline{m+1, 2m}.
$$
Using the relation
$$
( y^{[k]}, z ) = -( y^{[k-1]}, z' ) + (-1)^k \sum_{j = 0}^m ( q_{j,2m-k} y^{(j)}, z ), \quad z \in \mathfrak D,
$$
recursively for $k = n, n-1, \ldots, m+1$, we conclude that
\begin{equation} \label{sm1}
( y^{[n]}, z ) = (-1)^m ( y^{[m]}, z^{(m)} ) + \sum_{j = 0}^m \sum_{s = 0}^{m-1} ( q_{j,s} y^{(j)}, z^{(s)}).
\end{equation}
It follows from \eqref{ym} and \eqref{invS} that
\begin{equation} \label{sm2}
y^{[m]} = y^{(m)} + (-1)^m \sum_{j = 0}^{m-1} q_{j,m} y^{(j)}.
\end{equation}
Note that, in view of the definition in Lemma~\ref{lem:quad}, we have $q_{m,m} = 0$. Therefore, combining \eqref{sm1}, \eqref{sm2} and comparing the result with \eqref{quad}, we get 
$$
( \ell_n(y), z ) = ( y^{[n]}, z ), \quad z \in \mathfrak D. 
$$
Hence, $\ell_n(y) = y^{[n]}$ in $\mathfrak D'$.

On the other hand, taking \eqref{yk} for $k = n$ and \eqref{ym} into account, we conclude that $y^{[n]} \in L_{1,loc}(\mathbb R_+)$ for $y \in \mathcal D_F$. Thus, $\ell_n(y)$ is also a regular function, which completes the proof.
\end{proof}

For investigation of inverse spectral problems, we need the following technical lemma, which generalizes \cite[Lemma~2.1]{Bond21} and transfers it to the half-line case.

\begin{lem} \label{lem:PF-half}
Suppose that $\Sigma, \tilde \Sigma \in \Sigma_I$, $F = \mathscr F_I(\Sigma)$, $\tilde F = \mathscr F_I(\tilde \Sigma)$, and a unit lower-triangular matrix function $P(x) = [p_{k,j}(x)]_{k,j = 1}^n$ satisfies the equation
\begin{equation} \label{PF}
P'(x) + P(x) \tilde F(x) = F(x) P(x), \quad x \in \mathbb R_+.
\end{equation}

Then $P(x)$ on $\mathbb R_+$ identically equals the $(n \times n)$ unit matrix $I_n$ and $\Sigma = \tilde \Sigma$, that is, $\sigma_{\nu}(x) = \tilde \sigma_{\nu}(x)$ a.e. on $\mathbb R_+$, $\nu = \overline{0, n-2}$.
\end{lem}

\begin{proof}
\textsc{Step 1.} For definiteness, consider $n = 2m$. The case $n = 2m+1$ is analogous and even easier. Using the first $(m-1)$ rows of \eqref{PF}, we get $p_{k,j} = 0$ for $k = \overline{1, m}$, $j < k$. The $m$-th row of \eqref{PF} implies
\begin{equation} \label{pf1}
p_{m+1,j} = \tilde f_{m,j} - f_{m,j}, \quad j = \overline{1, m}.
\end{equation}
Similarly, considering the columns of \eqref{PF} for $j = 2m, 2m-1, \ldots, m+1$, we get $p_{k,j} = 0$ for $k > j$, $j = \overline{m+1, 2m}$, and
\begin{equation} \label{pf2}
p_{k,m} = f_{k,m+1} - \tilde f_{k,m+1}, \quad k = \overline{m+1,2m}.
\end{equation}
For $k = \overline{m+1,2m}$, $j = \overline{1, m}$, equation \eqref{PF} yields
\begin{equation} \label{pf3}
p'_{k,j} + p_{k,j-1} + p_{k,m} \tilde f_{m,j} + \tilde f_{k,j} = f_{k,j} + f_{k,m+1} p_{m+1, j} + p_{k+1, j}.
\end{equation}
Here, we assume that $p_{k,j} = 0$ if $j < 1$ or $k > 2m$.
Substituting \eqref{pf1} and \eqref{pf2} into \eqref{pf3}, we get
\begin{equation} \label{pf4}
p'_{k,j} + p_{k,j-1} + (\tilde f_{k,j} - \tilde f_{k,m+1} \tilde f_{m,j}) = p_{k+1,j} + (f_{k,j} - f_{k,m+1} f_{m,j}), \quad k = \overline{m+1, 2m}, \: j = \overline{1, m}.
\end{equation}
Using \eqref{invS}, pass to the new variables $[q_{l,s}]_{l,s = 0}^m := \mathscr S_n^{-1}(F)$, $[\tilde q_{l,s}]_{l,s = 0}^m := \mathscr S_n^{-1}(\tilde F)$ and $r_{j-1, 2m-k} := (-1)^{k+1} p_{k,j}$, $k = \overline{m,2m}$, $j = \overline{1, m+1}$. Thus, we get the system
\begin{gather} \label{eqr}
    r'_{l,s} + r_{l-1,s} + r_{l,s-1} = q_{l,s} - \tilde q_{l,s}, \quad l,s = \overline{0, m}, \quad (l,s) \ne (m,m), \\ \label{bcr}
    r_{m,s} = r_{s,m} = 0, \quad s = \overline{0, m-1}.
\end{gather}
Note that it is unimportant whether $r_{m,m-1} = r_{m-1,m} = 0$ or $r_{m,m-1} = r_{m-1,m} = 1$, since these values do not influence on the other entries.

In the case of odd $n$, we obtain exactly the same system \eqref{eqr}-\eqref{bcr} with respect to $r_{j-1,2m+1-k} := (-1)^k p_{k,j}$, $k = \overline{m+1, 2m+1}$, $j = \overline{1, m+1}$.

\smallskip

\textsc{Step 2.} 
Denote $\hat q_{l,s} := q_{l,s} - \tilde q_{l,s}$, $\hat \sigma_{\nu} = \sigma_{\nu} - \tilde \sigma_{\nu}$. 
It remains to prove that the relations \eqref{eqr}, \eqref{bcr} imply $r_{l,s}(x) = 0$ for $l, s = \overline{0, m}$, $(l,s) \ne (m,m)$, and $\hat \sigma_{\nu} = 0$, $\nu = \overline{0, n-2}$.
Let us show this by induction.
Suppose that we have already proved
$\hat \sigma_{2k} = 0$, $\hat \sigma_{2k+1} = 0$ for $k = \overline{0,K-1}$ with some fixed $K \in \{ 0, \ldots, m-1 \}$. This implies $\hat q_{k,s} = \hat q_{s,k} = 0$ for $s = \overline{0, m}$, $k = \overline{0, K-1}$. Therefore, it follows from \eqref{eqr}, \eqref{bcr} that $r_{k,s} = r_{s,k} = 0$ for $s = \overline{0, m}$, $k = \overline{0, K-1}$. Denote
$$
r_s^{\pm} := \frac{1}{2}(r_{K,s} \pm r_{s,K}), \quad \hat q_s^{\pm} := \frac{1}{2} (\hat q_{K,s} \pm \hat q_{s,K}).
$$
From \eqref{eqr},\eqref{bcr}, we get the systems
\begin{equation} \label{sysr}
(r_s^{\pm})' + r_{s-1}^{\pm} = \hat q_s^{\pm}, \quad s = \overline{0, m}, \quad r_m^{\pm} = 0. 
\end{equation}
By virtue of \eqref{defQ}, \eqref{defchi}, we have
$$
\hat q_s^+ = \left\{\begin{array}{ll}
                \hat \sigma_{2K}, & \: s = i_{2K} + K, \\
                0, & \: \text{otherwise},
            \end{array} \right.
\qquad            
\hat q_s^- = \left\{\begin{array}{ll}
                \hat \sigma_{2K+1}, & \: s = i_{2K+1} + K + 1, \\
                0, & \: \text{otherwise}.
            \end{array} \right.
$$
Therefore, considering \eqref{sysr} with ``$+$'' for $s = m, m-1, \ldots, i_{2K}+K+1$, we get $r_{s-1}^+ = 0$. Then, $r_{K+i_{2K}-1}^+ = \hat \sigma_{2K}$. If $i_{2K} = 0$, this immediately yields $\hat \sigma_{2K} = 0$. Otherwise, solving \eqref{sysr} for $s = K, K+1, \ldots, K+i_{2K}-1$, we obtain 
$$
r_{K+j}^+(x) = \sum_{l = 0}^j c_{j-l} \frac{(-1)^l x^l}{l!}, \quad j = \overline{0, i_{2K}-1},
$$
where $\{ c_l \}_{l = 0}^{i_{2K}-1}$ are arbitrary constants. Since $\sigma_{2K}$, $\tilde \sigma_{2K} \in L_1(\mathbb R_+)$, then $r_{K+i_{2K}-1}^+ \in L_1(\mathbb R_+)$, so $c_l = 0$, $l = \overline{0, i_{2K}-1}$. Hence, $r_s^+ = 0$ for $s = \overline{K, K+i_{2K}-1}$ and $\hat \sigma_{2K} = 0$.
Analogously, we show that $\hat \sigma_{2K+1} = 0$ by using the system \eqref{sysr} with ``$-$''. Note that, in the case $n = 2m$ and $K = m-1$, the expression \eqref{defl} does not contain the coefficient $\sigma_{2K+1}$, so the last step should be omitted.

Returning to the variables $p_{k,j}$, we arrive at the assertion of the lemma.
\end{proof}

In the case $\Sigma, \tilde \Sigma \in \Sigma_{I, (0,1)}$, 
$$
\Sigma_{I,(0,1)} := \left\{ (\sigma_{\nu})_{\nu = 0}^{n-2} \colon
\begin{array}{l}
\sigma_{\nu} \in L_1(0,1), \quad \nu = \overline{0, n-2}, \\
\sigma_{\nu} \in L_2(0,1), \quad \nu \in K(I)
\end{array}
\right\},
$$
the assertion of Lemma~\ref{lem:PF-half} is valid under additional initial conditions on $P(x)$. The following lemma generalizes and improves~\cite[Lemma~2.1]{Bond21}, and can be used for studying inverse spectral problems on a finite interval. 

\begin{lem} \label{lem:PF-finite}
Suppose that $\mathcal N \subseteq \{ 0, 1, \ldots, n-2 \}$, $\Sigma, \tilde \Sigma \in \Sigma_{I,(0,1)}$, $\sigma_{\nu}(x) = \tilde \sigma_{\nu}(x)$ a.e. on $(0,1)$ for $\nu \in \mathcal N$, $F = \mathscr F_I(\Sigma)$, $\tilde F = \mathscr F_I(\tilde \Sigma)$, and a unit lower-triangular matrix function $P(x) = [p_{k,j}(x)]_{k,j = 1}^n$ satisfies the equation \eqref{PF} on $(0,1)$ and the initial conditions $IC(\nu)$, $\nu = \overline{0, n-2} \setminus \mathcal N$, where
\begin{equation}
\label{icP}
\left.
\begin{array}{ll}
IC(2k) \colon &
p_{n-s,k+1}(0) + p_{n-k,s+1}(0) = 0, \quad
s = \overline{k, k+ i_{2k}-1}, \quad k = \overline{0, m-1} \\
IC(2k+1) \colon & p_{n-s,k+1}(0) - p_{n-k,s+1}(0) = 0, \quad s = \overline{k+1, k+i_{2k+1}}, \quad k = \overline{0, m + \tau - 2}
\end{array}
\right\}
\end{equation}

Then $P(x) \equiv I_n$ and $\Sigma = \tilde \Sigma$.
\end{lem}

\begin{proof}
Lemma~\ref{lem:PF-finite} is proved analogously to Lemma~\ref{lem:PF-half}. Step~1 requires no modifications. For simplicity, assume that $\mathcal N = \varnothing$. 
At Step~2, we consider the system~\eqref{eqr}-\eqref{bcr} together with the initial conditions 
\begin{equation} \label{icr}
\left.
\begin{array}{l}
    (r_{k,s} + r_{s,k})(0) = 0, \quad k = \overline{0, m-1}, \quad s = \overline{k, k + i_{2k} - 1} \\
    (r_{k,s} - r_{s,k})(0) = 0, \quad k = \overline{0, m + \tau - 2}, \quad s = \overline{k + 1, k + i_{2k+1}}
\end{array}\right\}
\end{equation}
which are equivalent to \eqref{icP}. Further, solving \eqref{sysr} with ``$+$'' for $s = K, K + 1, \ldots, K + i_{2K}-1$, we use the initial conditions $r_s^+(0) = 0$, which follow from \eqref{icr}. Therefore, we get $r_s^+ = 0$ for $s = \overline{K, K + i_{2K}-1}$, so $\hat \sigma_{2K} = 0$. The equality $\hat \sigma_{2K+1} = 0$ is proved analogously. Obviously, the proof is valid in the case $\mathcal N \ne \varnothing$ with minor modifications.
\end{proof}

\section{Birkhoff solutions} \label{sec:birk}

Suppose that $I \in \mathcal I_n$, $\Sigma \in \Sigma_I$, and $F = \mathscr F_I(\Sigma)$. In view of Theorem~\ref{thm:quasi}, we understand the solution of the equation \eqref{eqv}:
\begin{equation*} 
\ell_n(y) = \la y, \quad x \in \mathbb R_+,
\end{equation*}
in the following sense.

\begin{df} \label{def:sol}
A function $y$ is a \textit{solution} of equation \eqref{eqv} if $y \in \mathcal D_F$ and $\vec y$ satisfies the system \eqref{sys}:
$$
{\vec y}\,' = (F(x) + \Lambda) \vec y, \quad x \in \mathbb R_+.
$$
\end{df}

In this section, we obtain the Birkhoff solutions with the known behavior as $\rho \to \infty$ of equation \eqref{eqv} with $\la = \rho^n$.
Consider the partition of the $\rho$-plane into the sectors 
\begin{equation} \label{defGa}
\Gamma_k = \left\{ \rho \colon \frac{\pi(k-1)}{n} < \arg \rho < \frac{\pi k}{n} \right\}, \quad k = \overline{1, 2n}.
\end{equation}
Below, we assume that, if $\rho$ lies in a fixed sector $\Gamma = \Gamma_k$, then the roots $\{ \om_j \}_{j = 1}^n$ of the equation $\omega^n = 1$ are numbered so that
\begin{equation} \label{order}
\mbox{Re} \, (\rho \om_1) < \mbox{Re} \, (\rho \om_2) < \dots < \mbox{Re} \, (\rho \om_n), \quad \rho \in \Gamma.
\end{equation}

Put $\Omega := [\om_k^{j-1}]_{j,k = 1}^n$, $\la = \rho^n$. Applying the change of variables 
$$
\vec y(x) = \diag \{ 1, \rho, \ldots, \rho^{n-1} \} \Omega v(x) 
$$
(see \cite{SS20, Bond21} for details), we reduce the system \eqref{sys} to the form
\begin{gather} \label{ABD}
v' = \rho B v + A(x) v + D(x, \rho) v, \quad x \in \mathbb R_+, \\ \nonumber
B := \diag \{ \om_1, \om_2, \ldots, \om_n \}, \quad D(x, \rho) = \sum_{k = 1}^{n-1} \rho^{-k} D_k(x), 
\end{gather}
where $A(x)$ and $D_k(x)$, $k = \overline{1, n-1}$, are $(n \times n)$ matrix functions with entries of the classes $(L_2 \cap L_1)(\mathbb R_+)$ and $L_1(\mathbb R_+)$, respectively.

The Birkhoff solutions of differential systems generalizing \eqref{ABD} on a finite interval have been constructed in \cite{Rykh99, SS20, Yur-posobie}. Savchuk and Shkalikov \cite{SS20} have used those results to obtain the Birkhoff fundamental systems of solutions (FSS) of even-order differential equations with distribution coefficients. For the case of the half-line, their proofs are also valid with necessary modifications. However, the finite interval and the half-line cases have some important differences. In order to study inverse spectral problems for higher-order operators or differential systems on a finite interval, it is sufficient to have the Birkhoff solutions analytic for $\rho \in \Gamma$, $|\rho| > \rho_*$  with some fixed $\rho_* > 0$. For the half-line case, Yurko \cite{Yur02} has used the family of the Birkhoff FSS analytic for $\rho \in \Gamma$, $|\rho| > \rho_{\al}$ and depending on the parameter $\al \ge 0$, where $\rho_{\al} \to 0$ as $\al \to \infty$. Such FSS allowed him to study the properties of the spectral characteristics in the neighborhood of $\la = 0$ for the higher-order differential operators and the first-order differential systems. The construction of such Birkhoff systems for the case of regular coefficients is described, e.g., in \cite{Yur-posobie}. Developing the methods of \cite{Yur-posobie} for the system \eqref{ABD}, we have proved the following theorem.

\begin{thm} \label{thm:BirkY1}
For every $\alpha \ge 0$, there exists a FSS $\{ Y_{k,\al}(x, \rho) \}_{k = 1}^n$ of \eqref{ABD} having the following properties:
\begin{enumerate}
\item $Y_{k,\alpha}(x, \rho)$ are continuous for $x \in [0, \infty)$, $\rho \in \bar \Gamma$, $|\rho| \ge \rho_{\alpha}$.
\item For each $x \in [0, \infty)$, $Y_{k,\alpha}(x, \rho)$ are analytic in $\rho \in \Gamma$, $|\rho| \ge \rho_{\alpha}$.
\item The asymptotic relations
\begin{equation} \label{asymptY}
   Y_{k,\alpha}(x, \rho) = \exp(\rho \om_k x) (e_k + o(1)), \quad |\rho| \to \infty,
\end{equation} 
hold uniformly with respect to $x \ge \alpha$, $\rho \in \bar \Gamma$, where $e_k$ is the $k$-th column of the unit matrix $I_n$.
\end{enumerate}
\end{thm}

In Theorems~\ref{thm:BirkY1}-\ref{thm:Birkz}, it is supposed that  $\lim\limits_{\alpha \to \infty} \rho_{\alpha} = 0$.

Theorem~\ref{thm:BirkY1} and the change of variables 
$$
\vec y_{k,\al}(x) = \diag \{ 1, \rho, \ldots, \rho^{n-1} \} \Omega Y_{k,\al}(x), \quad k = \overline{1, n},
$$
readily imply the following result for equation~\eqref{eqv}.

\begin{thm} \label{thm:Birky1}
For every $\alpha \ge 0$, there exists FSS $\{ y_{k,\alpha}(x, \rho) \}_{k = 1}^n$ of equation \eqref{eqv} such that the quasi-derivatives $y_{k, \al}^{[j]}(x,\rho)$ for $k = \overline{1, n}$, $j = \overline{0, n-1}$ have the following properties:
\begin{enumerate}
\item $y_{k,\al}(x, \rho)$ are continuous for $x \in [0, \infty)$, $\rho \in \bar{\Gamma}$, $|\rho| \ge \rho_{\al}$.
\item For each $x \in [0, \infty)$, $y_{k,\al}^{[j]}$ are analytic in $\rho \in \Gamma$, $|\rho| \ge \rho_{\alpha}$;
\item The asymptotic relation
\begin{equation} \label{asympty}
y_{k,\al}^{[j]}(x, \rho) = (\rho \om_k)^j \exp(\rho \om_k x) (1 + o(1)), \quad |\rho| \to \iy, 
\end{equation}
holds uniformly with respect to $x \ge \al$ and $\rho \in \bar \Gamma$.
\end{enumerate}
\end{thm}

Fix $k \in \{ 1, \ldots, n-1\}$ and consider the region
\begin{equation} \label{defGk}
  G_k = \Bigl\{ \rho \in \mathbb C \colon \arg \rho \in \left((-1)^{n-k} - 1) \tfrac{\pi}{2n}, ((-1)^{n-k} + 3) \tfrac{\pi}{2n}\right) \Bigr\},
\end{equation}
being the union of two neighboring sectors $\Gamma$. Note that, while passing the boundary between two neighboring sectors, some neighboring values $\omega_j$ and $\omega_{j+1}$ are swapped in \eqref{order}.
The pair of sectors $G_k$ defined by \eqref{defGk} is chosen is such a way that $\om_k$ and $\om_{k+1}$ do not change their relative order, in other words, the sets $\{ \om_j \}_{j = 1}^k$ and $\{ \om_j \}_{j = k+1}^n$ are preserved in $G_k$.

Analogously to the system $B_{\alpha m}$ in Section~2.1.2 of \cite{Yur02}, we construct the following FSS.

\begin{thm} \label{thm:Birkz}
For every $\al \ge 0$ and $k \in \{ 1, \ldots, n-1 \}$, there exist solutions $\{ z_{s,k, \al} \}_{s = 1}^k$ of equation \eqref{eqv} with the quasi-derivatives $z_{s,k,\al}^{[j]}(x, \rho)$, $s = \overline{1, k}$, $j = \overline{0, n-1}$, having the following properties:
\begin{enumerate}
\item $z_{s,k,\alpha}^{[j]}(x, \rho)$ are continuous for $x \in [0, \infty)$, $\rho \in \bar G_k$, $|\rho| \ge \rho_{\alpha}$.
\item For each $x \in [0, \infty)$, $z_{s,k,\alpha}^{[j]}(x, \rho)$ are analytic in $\rho \in G_k$, $|\rho| \ge \rho_{\alpha}$.
\item The following uniform estimates hold:
$$
   z_{s,k,\al}^{[j]}(x, \rho) = O(\rho^j \exp(\rho \om_k x)), \quad x \ge \al, \quad \rho \in \bar G_k, \quad |\rho| \to \infty.
$$
\item In each of the two sectors $\Gamma \subset G_k$, the functions $\{ z_{1,k,\al}, \ldots, z_{k,k,\al}, y_{k+1,\al}, \ldots, y_{n,\al} \}$
form a FSS of equation \eqref{eqv}, where $y_{s,\al}$ are the solutions from Theorem~\ref{thm:Birky1}.
\end{enumerate}
\end{thm}

\section{Inverse problem on the half-line} \label{sec:ip}

Let $I \in \mathcal I_n$ and $\Sigma \in \Sigma_I$ be fixed. By using the matrix function $F = \mathscr F_I(\Sigma)$, define the quasi-derivatives $y^{[j]}$, $j = \overline{0, n}$, by \eqref{quasi}. Consider the boundary value problem $\mathcal L(\Sigma, U)$ given by \eqref{eqv}-\eqref{bc} and its Weyl matrix $M(\la) = [M_{s,k}(\la)]_{s,k = 1}^n$.
Using the Birkhoff systems constructed in Section~\ref{sec:birk}, we obtain the properties of the Weyl matrix, similar to the ones in the case of regular coefficients (see \cite[Theorem~2.1.1]{Yur02}).

\begin{thm} \label{thm:Weyl}
For each index pair $(s,k)$: $1 \le k < s \le n$, the Weyl function $M_{s,k}(\la)$ is analytic in $\Pi_{(-1)^{n-k}} := \mathbb C \setminus \{ \la \colon (-1)^{n-k} \la \ge 0 \}$ except for an at most countable bounded set of poles. For $(-1)^{n-k} \la \ge 0$ except for a bounded set, there exist finite limits $M_{s,k}^{\pm} = \lim\limits_{z \to 0, \, \mbox{Re}\, z > 0} M_{s,k}(\la \pm \mathrm{i} z)$.
\end{thm}

The proof of Theorem~\ref{thm:Weyl} repeats the proof of Theorem~2.1.1 in \cite{Yur02}, so we sketch it briefly.

\begin{proof}
For brevity, denote by $y_l := y_{l,0}$, $l = \overline{1,n}$, the Birkhoff solutions from Theorem~\ref{thm:Birky1} with $\alpha = 0$ in a fixed sector $\Gamma$.
In view of \eqref{bcPhi1}, the Weyl solutions $\Phi_k(x, \la)$ can be expanded as
\begin{equation} \label{Phiay}
\Phi_k(x, \la) = \sum_{l = 1}^k a_{k,l}(\rho) y_l(x, \rho), \quad k = \overline{1, n},
\end{equation}
where
\begin{equation} \label{defa}
a_{k,l}(\rho) := (-1)^{k+l} \frac{\det [U_j(y_r)]_{j = \overline{1, k-1}, \, r = \overline{1, k} \setminus l}}{\det [U_j(y_r)]_{j,r = \overline{1, k}}}.
\end{equation}
Since $M_{s,k}(\la) = U_s(\Phi_k)$, we obtain
\begin{equation} \label{relM}
M_{s,k}(\la) = \frac{\det [U_j(y_r)]_{j = \overline{1,k-1}, s, \, r = \overline{1,k}}}{\det [U_j(y_r)]_{j,r = \overline{1, k}}}.
\end{equation}

The similar arguments can be repeated for the Birkhoff systems $\{ y_l \}_{l = 1}^n := \{ y_{l, \alpha} \}_{l = 1}^n$ and $\{ y_l \}_{l = 1}^n := \{ z_{1,k,\al}, \ldots, z_{k,k,\al}, y_{k+1,\al}, \ldots, y_{n,\al} \}$ from Theorems~\ref{thm:Birky1} and \ref{thm:Birkz}, respectively. The relations of form \eqref{relM} together with the properties of solutions established in Theorems~\ref{thm:Birky1} and~\ref{thm:Birkz} yield the assertion of the theorem.
\end{proof}

Using \eqref{Phiay}, \eqref{defa} and the asymptotics \eqref{asympty}, we prove the following lemma.

\begin{lem} \label{lem:asymptPhi}
For each each fixed $x > 0$ and $\varphi$ such that $\{ \rho \colon \arg \rho = \varphi \} \subset \Gamma$, the following asymptotic relation holds
$$
\Phi_k^{[j]}(x, \la) = \rho^{-p_k} a_{k,k}^0 (\rho \om_k)^j \exp(\rho \om_k x) (1 + o(1)), \quad |\rho| \to \iy, \quad k = \overline{1, n}, \quad j = \overline{0, n-1},
$$
where $a_{k,k}^0 := \dfrac{d_{k-1,k-1}}{d_{k,k}} \ne 0$, $d_{k,k} = \det[\om_l^{p_s}]_{l,s = \overline{1,k}}$, $k = \overline{1, n}$, $d_{0,0} := 1$.
\end{lem}

Along with the problem $\mathcal L = \mathcal L_I(\Sigma, U)$, consider another problem $\tilde{\mathcal L} = \mathcal L_I(\tilde \Sigma, \tilde U)$ of the same form but with different coefficients $\tilde \Sigma$, $\tilde U$. We agree that, if a symbol $\ga$ denotes an object related to $\mathcal L$, then the symbol $\tilde \ga$ with tilde denotes the analogous object related to $\tilde{\mathcal L}$. In particular, $U = P_U L_U$, $\tilde U = P_{\tilde U} L_{\tilde U}$. 

\begin{thm} \label{thm:uniq}
If $P_U = P_{\tilde U}$ and $M(\la) \equiv \tilde M(\la)$, then $\Sigma = \tilde \Sigma$ (i.e. $\sigma_{\nu}(x) = \tilde \sigma_{\nu}(x)$ a.e. on $\mathbb R_+$, $\nu = \overline{0, n-2}$) and $U = \tilde U$.
\end{thm}

\begin{proof}
The proof is based on the method of spectral mappings (see \cite{Yur02, Bond21}).
Define the matrix of spectral mappings
\begin{equation} \label{defP}
\mathcal P(x, \la) := \Phi(x, \la) (\tilde \Phi(x, \la))^{-1}.
\end{equation}
Using the relations
$$
\Phi(x, \la) = C(x, \la) M(\la), \quad \tilde \Phi(x, \la) = \tilde C(x, \la) \tilde M(\la),
$$
and $M(\la) \equiv \tilde M(\la)$, we obtain
\begin{equation} \label{PC}
\mathcal P(x, \la) = C(x, \la) (\tilde C(x, \la))^{-1}.
\end{equation}

Due to the definition of $C(x, \la)$, this matrix function solves the initial value problem
$$
C'(x, \la) = (F(x) + \Lambda) C(x, \la), \quad x \in \mathbb R_+, \quad C(0, \la) = U^{-1}.
$$
Hence, $C(x, \la)$ is entire in $\la$ for each fixed $x \in \mathbb R_+$.
Since $\mbox{trace}\,(F(x)) \equiv 0$, then $\det(C(x, \la))$ does not depend on $x$. In view of the initial condition $C(0,\la) = U^{-1}$, we have $\det(C(x, \la)) \ne 0$. The same arguments are valid for $\tilde C(x, \la)$. Therefore, \eqref{PC} implies that $\mathcal P(x, \la)$ is entire in $\la$ for each fixed $x \in \mathbb R_+$ and $\mathcal P(0, \la) = U^{-1} \tilde U$. 

On the other hand, using \eqref{defP} and the asymptotics of Lemma~\ref{lem:asymptPhi} for the entries of $\Phi(x, \la)$ and $\tilde \Phi(x, \la)$, we obtain the following asymptotic relation for the entries of $\mathcal P(x, \la) = [\mathcal P_{k,j}]_{k,j = 1}^n$:
$$
\mathcal P_{k,j}(x, \la) = \rho^{k-j} (\de_{k,j} + o(1)), \quad k, j = \overline{1, n},
$$
for each fixed $x > 0$ as $|\la| \to \infty$ along any fixed ray $\arg \la = \beta \not\in \{ 0, \pi \}$. Applying Phragmen-Lindel\"of's theorem (see \cite{BFY14}) and Liouville's theorem, we conclude that $\mathcal P(x, \la)$ equals a constant unit lower-triangular matrix $\mathcal P(x)$ for each fixed $x > 0$ and 
\begin{equation} \label{P0}
\mathcal P(0) = U^{-1} \tilde U.
\end{equation}

Using \eqref{defP} and the relations
$$
\Phi'(x, \la) = (F(x) + \Lambda) \Phi(x, \la), \quad 
\tilde \Phi'(x, \la) = (\tilde F(x) + \Lambda) \tilde \Phi(x, \la),
$$
we derive
\begin{equation} \label{PF2}
\mathcal P'(x) + \mathcal P(x) \tilde F(x) = F(x) \mathcal P(x).
\end{equation}
Hence, $\mathcal P(x)$ satisfies the assumptions of Lemma~\ref{lem:PF-half}, which yields $\Sigma = \tilde \Sigma$ and $\mathcal P(x) \equiv I_n$. Using \eqref{P0}, we conclude that $U = \tilde U$.
\end{proof}

\begin{remark} \label{rem:arbF}
In fact, instead of the boundary value problem~\eqref{eqv},\eqref{bc}, we study the first-order system \eqref{sys} with the boundary condition $U \vec y(0) = 0$. Nevertheless, in the proof of Theorem~\ref{thm:uniq}, the special structure of the matrix $F = \mathscr F_I(\Sigma)$ constructed by the coefficients $\Sigma$ of the differential expression \eqref{defl} is important. An arbitrary matrix function $F \in \mathfrak F_n$ cannot be uniquely recovered from the corresponding Weyl matrix, because the assertion of Lemma~\ref{lem:PF-half} does not hold for an arbitrary $F \in \mathfrak F_n$. This is shown by Example~\ref{ex:1}. 
\end{remark}

\begin{example} \label{ex:1}
Suppose that $n = 2$, $F, \tilde F \in \mathfrak F_2$,
\begin{equation} \label{Fab}
F = \begin{bmatrix}
        a & 1 \\
        b & -a
    \end{bmatrix}, \quad
\tilde F = \begin{bmatrix}
                \tilde a & 1 \\
                \tilde b & -\tilde a
            \end{bmatrix},
\end{equation}
and a unit lower-triangular matrix function $P(x)$ satisfies \eqref{PF}. It is easy to see that the system \eqref{PF} in our case is equivalent to
$$
p_{2,1} = \tilde a - a, \quad p_{2,1}' + \tilde a^2 - a^2 + \tilde b - b = 0.
$$
The latter relations do not imply $p_{2,1} = 0$, $a = \tilde a$, $b = \tilde b$. 
For instance, one can take $a = b = 0$, an arbitrary function $\tilde a \in (L_1 \cap AC_{loc})(\mathbb R_+)$ such that $\tilde a' \in L_1(\mathbb R_+)$, $p_{21} := \tilde a$, $\tilde b := -(p_{21}' + p_{21}^2)$. Thus, the condition \eqref{PF} is fulfilled but the assertion of Lemma~\ref{lem:PF-half} does not hold for this case. The matrix function $F(x)$ of form \eqref{Fab} defines the quasi-derivatives 
$$
y^{[1]} = y' - a y, \quad y^{[2]} = (y^{[1]})' + a y^{[1]} - b y.
$$
Hence, the equation $y^{[2]} = \la y$ turns into the Sturm-Liouville equation $y'' - q(x) y = \la y$ with the potential $q = a' + a^2 + b$. Even if we reconstruct the potential $q$ by using some spectral data, we cannot uniquely determine the functions $a$ and $b$.
\end{example}

\section{Inverse problem on a finite interval} \label{sec:finite}

The inverse spectral problem for the differential expression \eqref{defl} on the finite interval $(0, 1)$ has been considered in \cite{Bond21} for the Mirzoev-Shkalikov case: $i_{2k+j} = m - k - j$, $j = 0, 1$. In contrast to the half-line, for a finite interval $W_1^{-{k_1}}(0, 1) \subset W_1^{-{k_2}}(0,1)$ if $k_1 < k_2$. Therefore, the case of arbitrary $I \in \mathcal I_n$ can be reduced to the Mirzoev-Shkalikov case, and the results of \cite{Bond21} can be applied. However, in this section, we show that the regularization of Section~\ref{sec:reg} for any $I \in \mathcal I_n$ can be used for investigating inverse problems. In addition, we discuss the recovery of the boundary conditions, improving the results of \cite{Bond21}.

Suppose that $I \in \mathcal I_n$, $\Sigma \in \Sigma_{I, (0,1)}$.
Denote by $\mathcal L = \mathcal L_I(\Sigma, U, V)$ the differential equation
\begin{equation} \label{eqv-finite}
\ell_n(y) = \la y, \quad x \in (0, 1),
\end{equation}
given together with the linear forms \eqref{bc} and
$$
V_s(y) := y^{[p_{s,1}]}(1) + \sum_{j = 1}^{p_{s,1}} v_{s,j} y^{[j-1]}(1), \quad s = \overline{1, n},
$$
where $V = [v_{s,j}]_{s,j = 1}^n$ is a constant $(n \times n)$ matrix of form $V = P_V L_V$, $P_V$ is the permutation matrix with the unit elements at the positions $(k, p_{k,1} + 1)$, $k = \overline{1, n}$, and $L_V$ is a unit lower-triangular matrix.

Denote by $\{ C_k(x, \la) \}_{k = 1}^n$ and $\{ \Phi_k(x, \la) \}_{k = 1}^n$ the solutions of equation \eqref{eqv-finite} satisfying the conditions \eqref{initC1} and
\begin{equation} \label{bcPhi2}
U_s(\Phi_k) = \de_{s,k}, \quad s = \overline{1, k}, \qquad
V_l(\Phi_k) = 0, \quad l = \overline{k+1,n},
\end{equation}
respectively. Define the matrix functions $C(x, \la) = [\vec C_k(x, \la)]_{k = 1}^n$ and $\Phi(x, \la) = [\vec \Phi_k(x, \la)]_{k = 1}^n$. Then, $C(x, \la) = \Phi(x, \la) M(\la)$, where $M(\la)$ is \textit{the Weyl matrix}. It is shown in \cite{Bond21} that $M(\la)$ is a unit lower-triangular matrix function meromorphic in $\la$.

It has been proved in \cite{Bond21} that the Weyl matrix $M(\la)$ uniquely specifies the coefficients $\Sigma$ in the Mirzoev-Shkalikov case if the matrices $U$ and $V$ are known a priori. Here, we focus on the recovery of the boundary conditions in more details for various $I \in \mathcal I_n$.

Using the entries of the matrix $L_U = [l_{k,j}]_{k,j = 1}^n$, define the vectors
$$
\begin{array}{ll}
L_{2k} & := (l_{n-s,k+1} + l_{n-k, s+1} \colon s = \overline{k, k + i_{2k}-1}), \quad k = \overline{0, m-1}, \\
L_{2k+1} & := (l_{n-s,k+1} - l_{n-k, s+1} \colon s = \overline{k+1, k + i_{2k+1}}), \quad k = \overline{0, m + \tau -2}.
\end{array}
$$

\begin{ip} \label{ip:finite}
Suppose that $P_U$ and $L_{\nu}$, $\nu = \overline{0, n-2}$, are known a priori.
Given the Weyl matrix $M(\la)$, find $\Sigma$, $U$, and $V$.
\end{ip}

Note that, in the regular case $i_{\nu} = 0$, $\nu = \overline{0,n-2}$, no elements of $L_U$ are required to be known in Inverse Problem~\ref{ip:finite}. In the case $i_{2k} = i_{2k+1} + 1 =: j_k$, the values $L_{2k}$ and $L_{2k+1}$ can be replaced by
$l_{n-s, k + 1}$ and $l_{n-k,s+1}$, $s = \overline{k, k + j_k - 1}$. In particular, in the Mirzoev-Shkalikov case, the values $\{ l_{k,j} \}_{k = \overline{m + \tau + 1, n}, j = \overline{1, m}}$ are required to be known.

Along with the problem $\mathcal L = \mathcal L_I(\Sigma, U, V)$, consider another problem $\tilde{\mathcal L} = \mathcal L_I(\tilde \Sigma, \tilde U, \tilde V)$ of the same form but with different coefficients $\tilde \Sigma$, $\tilde U$, $\tilde V$. We agree that, if a symbol $\ga$ denotes an object related to $\mathcal L$, then the symbol $\tilde \ga$ with tilde denotes the analogous object related to $\tilde {\mathcal L}$. 

In view of Remark~4.1 in \cite{Bond21}, the right-hand boundary condition coefficients cannot be uniquely recovered from the Weyl matrix. However, some equivalence classes can be considered, so we need the following definition.

\begin{df} \label{def:V}
Let $I \in \mathcal I_n$, $\Sigma \in \Sigma_{I, (0,1)}$, and $U = P_U L_U$ be fixed. Then the matrices $V = P_V L_V$ and $\tilde V = P_{\tilde V} L_{\tilde V}$ are called \textit{equivalent} if the corresponding problems $\mathcal L = \mathcal L(\Sigma, U, V)$ and $\tilde {\mathcal L} = \mathcal L(\Sigma, U, \tilde V)$ have equal Weyl solutions: $\Phi_k(x, \la) \equiv \tilde \Phi_k(x, \la)$, $k = \overline{1, n}$.
\end{df}

The following uniqueness theorem generalizes Theorem~6.2 from \cite{Bond21}.

\begin{thm} \label{thm:finite}
If $P_U = P_{\tilde U}$, $L_{\nu} = \tilde L_{\nu}$, $\nu = \overline{0, n-2}$, and $M(\la) \equiv \tilde M(\la)$, then $\Sigma = \tilde \Sigma$ (i.e. $\sigma_{\nu}(x) = \tilde \sigma_{\nu}(x)$ a.e. on $(0,1)$, $\nu = \overline{0, n-2}$), $U = \tilde U$, and $V$ is equivalent to $\tilde V$ in sense of Definition~\ref{def:V}.
\end{thm}

The proof of Theorem~\ref{thm:finite} is analogous to the proof of Theorem~6.2 in \cite{Bond21} and relies on Lemma~\ref{lem:PF-finite}, so we omit it.

\begin{remark} \label{rem:uniqN}
Suppose that $\mathcal N \subseteq \{ 0, 1, \ldots, n-2 \}$ and the functions $( \sigma_{\nu} )_{\nu \in \mathcal N}$ are known a priori. Then, it is sufficient to know $L_{\nu}$ for $\nu = \overline{0, n-2} \setminus \mathcal N$ together with $P_U$ for the unique recovery of the problem $\mathcal L(\Sigma, U, V)$ from $M(\la)$.
\end{remark}

\begin{remark}
In contrast to Inverse Problem~\ref{ip:finite}, Inverse Problem~\ref{ip:main} on the half-line does not require the boundary condition coefficients $l_{k,j}$ to be known, roughly speaking, because of the implicit condition at infinity: $\sigma_{\nu} \in L_1(\mathbb R_+)$, $\nu = \overline{0, n-2}$.
\end{remark}

\section{Examples} \label{sec:ex}

\subsection{Case n = 2.}

The differential expression \eqref{defl} for $n = 2$ takes the form
$$
\ell_2(y) = y'' + (-1)^{i_0} \sigma_0^{(i_0)} y,
$$
where $i_0 \in \{ 0, 1 \}$, $I = (i_0)$, $\Sigma = (\sigma_0)$. First, consider the inverse problem on \textit{the half-line}.

\medskip

{\bf 1.} In the case $i_0 = 0$, equation \eqref{eqv} takes the form
\begin{equation} \label{eq20}
y'' + \sigma_0 y = \la y, \quad \sigma_0 \in L_1(\mathbb R_+).
\end{equation}
Using \eqref{defQ}, \eqref{defchi}, and \eqref{defF}, we obtain the matrix functions $Q = \mathscr Q_I(\Sigma)$, $F = \mathscr F_I(\Sigma)$:
$$
Q = \begin{bmatrix}
        \sigma_0 & 0 \\ 0 & 0
    \end{bmatrix},
\qquad
F = \begin{bmatrix}
        0 & 1 \\ -\sigma_0 & 0
    \end{bmatrix}.
$$
Hence, $y^{[1]} = y'$, $y^{[2]} = y'' + \sigma_0 y$. For definiteness, suppose that
\begin{equation} \label{PL2}
P_U = \begin{bmatrix}
        0 & 1 \\
        1 & 0
     \end{bmatrix},
\qquad
L_U = \begin{bmatrix}
        1 & 0 \\ h & 1
    \end{bmatrix}, \quad h \in \mathbb C.
\end{equation}
Then
$$
U_1(y) = y'(0) + h y(0), \quad U_2(y) = y(0).
$$
The Weyl matrix has the form 
\begin{equation} \label{M2}
M(\la) = \begin{bmatrix}
            1 & 0 \\
            M_{2,1}(\la) & 1
        \end{bmatrix},
\quad M_{2,1}(\la) = U_2(\Phi_1),
\end{equation}
where $\Phi_1(x, \la)$ is the Weyl solution of equation \eqref{eq20} satisfying the boundary conditions
\begin{equation} \label{Phi2}
U_1(\Phi_1) = 1, \quad \Phi_1(x, \la) = O(\exp(-\rho x)), \quad x \to \infty, \quad \mbox{Re}\,\rho \ge 0, \: \rho \ne 0.
\end{equation}
Inverse Problem~\ref{ip:main} takes the following form.

\begin{ip} \label{ip:2}
Given the Weyl function $M_{2,1}(\la)$, find $\sigma_0$ and $h$.
\end{ip}

This is the standard inverse problem for the Sturm-Liouville operator on the half-line by the Weyl function, which has been considered, e.g., in \cite[Section~2.2]{FY01}. Theorem~\ref{thm:uniq} for this case is equivalent to Theorem~2.2.1 in \cite{FY01}.

\medskip

{\bf 2.} In the case $i_0 = 1$, equation \eqref{eqv} takes the form
\begin{equation} \label{eq21}
y'' - \sigma_0' y = \la y, \quad \sigma_0 \in (L_1 \cap L_2)(\mathbb R_+),
\end{equation}
where the derivative is understood in the sense of distributions.

Using \eqref{defQ}, \eqref{defchi}, and \eqref{defF}, we obtain the matrix functions $Q = \mathscr Q_I(\Sigma)$, $F = \mathscr F_I(\Sigma)$:
$$
Q = \begin{bmatrix}
        0 & \sigma_0 \\ \sigma_0 & 0
    \end{bmatrix},
\qquad
F = \begin{bmatrix}
        \sigma_0 & 1 \\ -\sigma_0^2 & -\sigma_0
    \end{bmatrix},
$$
Thus, $F(x)$ coincides with the well-known regularization matrix for the Sturm-Liouville operator with singular potential (see, e.g., \cite{Pfaff79, SS99}).
The quasi-derivatives have the form
$$
y^{[1]} = y' - \sigma_0 y, \quad y^{[2]} = (y^{[1]})' + \sigma_0 y^{[1]} + \sigma_0^2 y.
$$

Define $P_U$ and $L_U$ by \eqref{PL2}. Then
$$
U_1(y) = y^{[1]}(0) + h y(0), \quad U_2(y) = y(0).
$$
The Weyl matrix has the form \eqref{M2}, where $\Phi_1(x, \la)$ is the solution of equation \eqref{eq21} (in the sense of Definition~\ref{def:sol}) satisfying the boundary conditions \eqref{Phi2}. Inverse Problem~\ref{ip:main} for this case takes the form of Inverse Problem~\ref{ip:2}.

Suppose that the problems $\mathcal L$ with $i_0 = 0$, $\sigma_0 \in L_1(\mathbb R_+)$ and $\tilde {\mathcal L}$ with $\tilde i_0 = 1$, $\tilde \sigma_0 \in (L_1 \cap L_2)(\mathbb R_+)$ are equivalent to each other. 
Let us show that the corresponding inverse problems are also equivalent to each other. Comparing \eqref{eq20} and \eqref{eq21}, we conclude that $\sigma_0 = -\tilde \sigma_0'$, so
\begin{equation} \label{relsi}
\tilde \sigma_0 \in AC[0, \infty), \quad \tilde \sigma_0(x) = \tilde \sigma_0(0) - \int_0^x \sigma_0(t) \, dt.
\end{equation}
Note that the linear forms $U_1$ and $\tilde U_1$ for the problems $\mathcal L$ and $\tilde {\mathcal L}$, respectively, differ:
$$
U_1(y) = y'(0) + h y(0), \quad \tilde U_1(y) = y^{[1]}(0) + \tilde h y(0) = y'(0) - \tilde \sigma_0(0) y(0) + \tilde h y(0).
$$
These forms coincide with each other and provide the same Weyl function $M_{2,1}(\la) = \tilde M_{2,1}(\la)$ if and only if
\begin{equation} \label{relh}
h = -\tilde \sigma_0(0) + \tilde h.
\end{equation}
The relations \eqref{relsi} and \eqref{relh} together imply
\begin{equation} \label{equiv2}
\tilde \sigma_0(x) = \tilde h - h - \int_0^x \sigma_0(t) \, dt.
\end{equation}
Since $\sigma_0, \tilde \sigma_0 \in L_1(\mathbb R_+)$, then 
\begin{equation} \label{inf2}
    \tilde h - h - \int_0^{\iy} \sigma_0(t) \, dt = 0.
\end{equation}
The relations~\eqref{equiv2} and \eqref{inf2} give the one-to-one correspondence between the data $(\sigma_0, h) \leftrightarrow (\tilde \sigma_0, \tilde h)$. Thus, the reconstruction of either $(\sigma_0, h)$ or $(\tilde \sigma_0, \tilde h)$ by using $M_{2,1}(\la)$ is equivalent.

The situation is different for \textit{a finite interval}. For definiteness, consider
$$
P_V = \begin{bmatrix}
            1 & 0 \\ 0 & 1
      \end{bmatrix}, \qquad
L_V = \begin{bmatrix}
            1 & 0 \\ H & 1
      \end{bmatrix}, \quad H \in \mathbb C.
$$
Then $V_2(y) = y^{[1]}(1) + H y(1)$. The Weyl matrix has the form \eqref{M2}, where $\Phi_1(x, \la)$ is the Weyl solution of the equation $\ell_2(y) = \la y$, $x \in (0, 1)$, satisfying the boundary conditions
$U_1(\Phi_1) = 1$, $V_2(\Phi_1) = 0$, where $U_1(y)$ is defined similarly to the half-line case. In the regular case $i_0 = 0$, Inverse Problem~\ref{ip:finite} takes the following form.

\begin{ip} \label{ip:20}
Given $M_{2,1}(\la)$, find $\sigma_0$, $h$, and $H$.
\end{ip}

Inverse Problem~\ref{ip:20} is the classical problem of the recovery of the Sturm-Liouville operator from the Weyl function, which is equivalent to Borg's problem by two spectra and to Marchenko's problem by the spectral function (see, e.g., \cite{Borg46, Mar77, FY01}).

In the singular case $i_1 = 0$, Inverse Problem~\ref{ip:finite} can be reformulated as follows.

\begin{ip} \label{ip:21}
Given $h$ and $M_{2,1}(\la)$, find $\sigma_0$ and $H$.
\end{ip}

Inverse Problem~\ref{ip:21} in various equivalent formulations was studied in \cite{HM-sd, HM-2sp, Gul19} and other papers. The uniqueness Theorem~\ref{thm:uniq} for Inverse Problems~\ref{ip:20} and~\ref{ip:21} corresponds to the previously known results.

Suppose that the problem $\mathcal L$ with $i_0 = 0$, $\sigma_0 \in L_1(0,1)$ is equivalent to the problem $\tilde {\mathcal L}$ with $i_0 = 1$, $\tilde \sigma_0 \in W_1^1[0, 1]$, that is, the relations \eqref{equiv2} and
\begin{equation} \label{H2}
    H = -\tilde \sigma_0(1) + \tilde H
\end{equation}
hold.
If $\tilde h$ is fixed, then \eqref{inf2} and \eqref{H2} give the one-to-one correspondence between the data $(\sigma_0, h, H) \leftrightarrow (\tilde \sigma_0, \tilde H)$. Consequently, Inverse Problem~\ref{ip:20} for $\mathcal L$ and Inverse Problem~\ref{ip:21} for $\tilde {\mathcal L}$ are equivalent to each other in this case. 

\subsection{Case n = 4.}

In \cite{Kon21}, the regularization matrices have been provided for the differential expression
\begin{equation} \label{lKon}
(py'')'' - (q^{(\al)} y')' + r^{(\be)} y,
\end{equation}
where $p, q, r$ are regular functions, $\al \in \{ 0, 1 \}$, $\be \in \{ 0, 1, 2 \}$. 
In the case $p \equiv 1$, \eqref{lKon} is equivalent to the differential expression \eqref{defl} for $n = 4$, $I = (i_0, i_1, i_2)$, $\Sigma = (\sigma_0, \sigma_1, \sigma_2)$ with $\sigma_1 = 0$:
$$
\ell_4(y) = y^{(4)} + (-1)^{i_2 + 1} (\sigma_2^{(i_2)}(x) y')' + (-1)^{i_0} \sigma_0^{(i_0)} y.
$$
Here $i_0 \in \{ 0, 1, 2 \}$, $i_2 \in \{ 0, 1 \}$. 
Let us consider all the six possible cases.
For convenience, denote $Q_{i_0, i_2} := \mathscr Q_I(\Sigma)$, $F_{i_0, i_2} := \mathscr F_I(\Sigma)$. 

Using \eqref{defQ}, \eqref{defchi}, and \eqref{defF}, we obtain
\begin{align*}
Q_{0,0} & = \begin{bmatrix}
            \sigma_0 & 0 & 0 \\
            0 & \sigma_2 & 0 \\
            0 & 0 & 0
        \end{bmatrix}, \qquad
Q_{1,0} = \begin{bmatrix}
            0 & \sigma_0 & 0 \\
            \sigma_0 & \sigma_2 & 0 \\
            0 & 0 & 0
         \end{bmatrix}, \qquad
Q_{2,0} = \begin{bmatrix}
            0 & 0 & \sigma_0 \\
            0 & \sigma_2 + 2 \sigma_0 & 0 \\
            \sigma_0 & 0 & 0
        \end{bmatrix},    \\
Q_{0,1} & = \begin{bmatrix}
            \sigma_0 & 0 & 0 \\
            0 & 0 & \sigma_2 \\
            0 & \sigma_2 & 0
         \end{bmatrix}, \qquad
Q_{1,1} = \begin{bmatrix}
             0 & \sigma_0 & 0 \\
             \sigma_0 & 0 & \sigma_2 \\
             0 & \sigma_2 & 0
          \end{bmatrix}, \qquad
Q_{2,1} = \begin{bmatrix}
            0 & 0 & \sigma_0 \\
            0 & 2 \sigma_0 & \sigma_2 \\
            \sigma_0 & \sigma_2 & 0
         \end{bmatrix}
\end{align*}
\begin{align*}
F_{0,0} & = \begin{bmatrix}
            0 & 1 & 0 & 0 \\
            0 & 0 & 1 & 0 \\
            0 & \sigma_2 & 0 & 1 \\
            -\sigma_0 & 0 & 0 & 0
         \end{bmatrix}, \quad
F_{1,0} = \begin{bmatrix}
            0 & 1 & 0 & 0 \\
            0 & 0 & 1 & 0 \\
            \sigma_0 & \sigma_2 & 0 & 1 \\
            0 & -\sigma_0 & 0 & 0
         \end{bmatrix}, \quad
F_{2,0} = \begin{bmatrix}
            0 & 1 & 0 & 0 \\
            -\sigma_0 & 0 & 1 & 0 \\
            0 & \sigma_2 + 2\sigma_0 & 0 & 1 \\
            \sigma_0^2 & 0 & -\sigma_0 & 0
         \end{bmatrix}, \\
F_{0,1} & = \begin{bmatrix}
            0 & 1 & 0 & 0 \\
            0 & -\sigma_2 & 1 & 0 \\
            0 & -\sigma_2^2 & \sigma_2 & 1 \\
            -\sigma_0 & 0 & 0 & 0
         \end{bmatrix}, \quad
F_{1,1} = \begin{bmatrix}
            0 & 1 & 0 & 0 \\
            0 & -\sigma_2 & 1 & 0 \\
            \sigma_0 & -\sigma_2^2 & \sigma_2 & 1 \\
            0 & -\sigma_0 & 0 & 0
         \end{bmatrix}, \quad
F_{2,1} = \begin{bmatrix}
            0 & 1 & 0 & 0 \\
            -\sigma_0 & -\sigma_2 & 1 & 0 \\
            -\sigma_0 \sigma_2 & 2 \sigma_0 - \sigma_2^2 & \sigma_2 & 1 \\
            \sigma_0^2 & \sigma_0 \sigma_2 & -\sigma_0 & 0
         \end{bmatrix}.
\end{align*}

The matrices $F_{i_0, i_2}$ coincide with the ones provided in \cite{Kon21}. In particular, the matrix $F_{0,0}$ corresponds to the well-known regular case (see \cite[Appendix A]{EM99}), and $F_{2,1}$ was obtained in \cite{Vlad04}.

For clarity, denote $\mathcal L_{i_0, i_2} := \mathcal L_I$.
Similarly to the case $n = 2$, it can be shown that, if the problem $\mathcal L = \mathcal L_{i_0, i_2}(\Sigma, U)$ on \textit{the half-line} is equivalent to $\tilde{\mathcal L} = \mathcal L_{\tilde i_0, \tilde i_2}(\tilde \Sigma, \tilde U)$ with $(\tilde i_0, \tilde i_2) \ne (i_0, i_2)$, $P_U = P_{\tilde U}$, then the corresponding inverse problems are equivalent to each other. We obtain the equivalence relations between the problem coefficients $(\Sigma, U) \leftrightarrow (\tilde \Sigma, \tilde U)$ analogous to \eqref{equiv2},\eqref{inf2} for several cases. The other cases can be investigated similarly.

\medskip

{\bf 1.} Consider equivalent problems $\mathcal L = \mathcal L_{0,0}(\Sigma, U)$, $\tilde {\mathcal L} = \mathcal L_{0,1}(\tilde \Sigma, \tilde U)$, where $\sigma_0 = \tilde \sigma_0 \in L_1(\mathbb R_+)$, $\sigma_2 = -\tilde \sigma_2'$, $\sigma_2 \in L_1(\mathbb R_+)$, $\tilde \sigma_2 \in (L_1 \cap L_2)(\mathbb R_+)$. The quasi-derivatives for the problems $\mathcal L$ and $\tilde {\mathcal L}$ are defined via \eqref{quasi} by using the entries of the matrix functions $F_{0,0}$ and $F_{0,1}$, respectively. Substituting these quasi-derivatives into the equivalence relations for the boundary condition forms: $U_s(y) = \tilde U_s(y)$, $s = \overline{1, 4}$, we derive
$$
l_{2,1} = \tilde l_{2,1}, \quad l_{3,1} = \tilde l_{3,1}, \quad l_{3,2} = \tilde \sigma_2(0) + \tilde l_{3,2}, \quad l_{4,1} = \tilde l_{4,1}, \quad l_{4,2} = \tilde l_{4,3} \tilde \sigma_2(0) + \tilde l_{4,2}, \quad l_{4,3} = \tilde l_{4,3}.
$$
Consequently, the equivalence $(\Sigma, U) \leftrightarrow (\tilde \Sigma, \tilde U)$ is
given by the relations
\begin{gather} \nonumber
l_{3,2} = \tilde l_{3,2} + \int_0^{\iy} \sigma_2(t) \, dt, \quad
\tilde \sigma_2(x) = \int_x^{\iy} \sigma_2(t) \, dt, \\ \label{case1}
(\sigma_0, l_{2,1}, l_{3,1}, l_{4,1}, l_{4,3}) = (\tilde \sigma_0, \tilde l_{2,1}, \tilde l_{3,1}, \tilde l_{4,1}, \tilde l_{4,3}), \quad l_{4,2} = \tilde l_{4,3} \tilde \sigma_2(0) + \tilde l_{4,2}.
\end{gather}

\medskip

{\bf 2.} Consider equivalent problems $\mathcal L = \mathcal L_{0,1}(\Sigma, U)$, $\tilde {\mathcal L} = \mathcal L_{1,1}(\tilde \Sigma, \tilde U)$, where $\sigma_0 = -\tilde \sigma_0'$, $\sigma_0 \in L_1(\mathbb R_+)$, $\tilde \sigma_0 \in L_1 (\mathbb R_+)$, $\sigma_2 = \tilde \sigma_2 \in (L_1 \cap L_2)(\mathbb R_+)$. The one-to-one correspondence $(\Sigma, U) \leftrightarrow (\tilde \Sigma, \tilde U)$
is given by the relations
\begin{gather} \nonumber
\tilde l_{4,1} = l_{4,1} + \int_0^{\iy} \sigma_0(t) \, dt, \quad
\tilde \sigma_0(x) = \int_x^{\iy} \sigma_0(t) \, dt, \\ \label{case2}
(\sigma_2, l_{2,1}, l_{3,1}, l_{3,2}, l_{4,2}, l_{4,3}) = (\tilde \sigma_2, \tilde l_{2,1}, \tilde l_{3,1}, \tilde l_{3,2}, \tilde l_{4,2}, \tilde l_{4,3}).
\end{gather}

\medskip

{\bf 3.} Consider equivalent problems $\mathcal L = \mathcal L_{1,0}(\Sigma, U)$, $\tilde {\mathcal L} = \mathcal L_{2,0}(\tilde \Sigma, \tilde U)$, where $\sigma_0 = -\tilde \sigma_0'$, $\sigma_0 \in L_1(\mathbb R_+)$, $\sigma_0$ is continuous at zero, $\tilde \sigma_0 \in (L_1 \cap L_2)(\mathbb R_+)$, $\sigma_2 = \tilde \sigma_2 \in L_1(\mathbb R_+)$. The one-to-one correspondence $(\Sigma, U) \leftrightarrow (\tilde \Sigma, \tilde U)$
is given by the relations
\begin{gather} \nonumber
l_{3,1} = \tilde l_{3,1} + \int_0^{\iy} \sigma_0(t) \, dt, \quad
\tilde \sigma_0(x) = \int_x^{\iy} \sigma_0(t) \, dt, \\ \label{case3}
\left. \begin{array}{c}
(\sigma_2, l_{2,1}, l_{3,2}, l_{4,3}) = (\tilde \sigma_2, \tilde l_{2,1}, \tilde l_{3,2}, \tilde l_{4,3}) \\ 
\quad l_{4,2} = \tilde l_{4,2} - 2 \tilde \sigma_0(0), \quad l_{4,1} - \sigma_0(0) = \tilde l_{4,3} \tilde \sigma_0(0) + \tilde l_{4,1}
\end{array}\right\}
\end{gather}

\medskip

Proceed to the \textit{finite interval} case.
Since $\sigma_1 = 0$ is known, put $\mathcal N := \{ 1 \}$. 
Taking Theorem~\ref{thm:uniq} and Remark~\ref{rem:uniqN} into account, we conclude that the numbers
\begin{equation} \label{exl}
l_{4-s, 1} + l_{4, s + 1}, \quad s = \overline{0, i_0-1},
\end{equation}
and $l_{3,2}$ if $i_2 = 1$ have to be given together with $P_U$ and $M(\la)$ for the unique reconstruction of $\mathcal L_I(\Sigma, U, V)$. Alternatively, one can give either $\{ l_{s,1} \}_{s = 4 - i_0 + 1}^{4}$ or $\{ l_{4,s} \}_{s = 1}^{i_0}$ instead of \eqref{exl}. For definiteness, suppose that we have $\{ l_{s,1} \}_{s = 4 - i_0 + 1}^{4}$.

Let us shortly denote by $\mbox{IP}_{i_0, i_2}$ the inverse problem for the corresponding $i_0$ and $i_2$. Suppose that $P_U$ and $M(\la)$ are given. For the recovery of $\Sigma$, $U$, and $V$ the following
boundary condition coefficients $l_{k,j}$ are required:
\begin{equation*}
    \mbox{IP}_{0,0}: \text{none}, \quad
    \mbox{IP}_{1,0}: l_{4,1}, \quad
    \mbox{IP}_{2,0}: l_{3,1}, l_{4,1} \quad
    \mbox{IP}_{0,1}: l_{3,2}, \quad
    \mbox{IP}_{1,1}: l_{3,2}, l_{4,1}, \quad
    \mbox{IP}_{2,1}: l_{3,1}, l_{3,2}, l_{4,1}.
\end{equation*}

It can be shown that, if the problem $\mathcal L_{i_0, i_2}(\Sigma, U, V)$ is equivalent to $\tilde {\mathcal L}_{\tilde i_0, \tilde i_2}$ with $(\tilde i_0, \tilde i_2) \ne (i_0, i_2)$, $P_U = P_{\tilde U}$, then the corresponding inverse problems $\mbox{IP}_{i_0, i_2}$ and $\mbox{IP}_{\tilde i_0, \tilde i_2}$ are equivalent to each other.

For simplicity, assume that the matrix $P_V$ is defined by the permutation $(p_{1,1}, p_{2,1}, p_{3,1}, p_{4,1}) = (3, 2, 1, 0)$. Thus, the Weyl solutions \eqref{bcPhi2} is defined by the following linear forms:
$$
V_2(y) = y^{[2]}(1) + v_{2,2} y^{[1]}(1) + v_{2,1} y(1), 
\quad V_3(y) = y^{[1]}(1) + v_{3,1} y(1), \quad V_4(y) = y(1).
$$
Note that the linear form $V_1(y)$ does not participate in \eqref{bcPhi2}. Moreover, the Weyl solutions $\Phi_k(x, \la)$ do not depend on the coefficients $(v_{2,1}, v_{2,2}, v_{3,1})$. Therefore, all matrices $V$ with the fixed $P_V$ are equivalent in the sense of Definition~\ref{def:V}. Hence, $L_V$ cannot be uniquely recovered from the Weyl matrix $M(\la)$ even if $\Sigma$ and $U$ are known. In order to prove the inverse problem equivalence in this case, we only need to obtain the equivalence relations $(\Sigma, U) \leftrightarrow (\tilde \Sigma, \tilde U)$. Below, we consider the cases 1-3 similar to the ones studied for the half-line.

\medskip

{\bf 1.} Consider equivalent problems $\mathcal L = \mathcal L_{0,0}(\Sigma, U, V)$, $\tilde L = \mathcal L_{0,1}(\tilde \Sigma, \tilde U, \tilde V)$, where $\sigma_0 = \tilde \sigma_0 \in L_1(0,1)$, $\sigma_2 = -\tilde \sigma_2' \in L_1(0,1)$. If $\tilde l_{3,2}$ is fixed, then the one-to-one correspondence
$$
(\sigma_0, \sigma_2, l_{2,1}, l_{3,1}, l_{3,2}, l_{4,1}, l_{4,2}, l_{4,3}) \leftrightarrow (\tilde \sigma_0, \tilde \sigma_2, \tilde l_{2,1}, \tilde l_{3,1}, \tilde l_{4,1}, \tilde l_{4,2}, \tilde l_{4,3})
$$
is given by \eqref{case1} and
$$
\tilde \sigma_2(x) = l_{3,2} - \tilde l_{3,2} - \int_0^x \sigma_2(t) \, dt.
$$
Hence, $\mbox{IP}_{0, 0}$ is equivalent to $\mbox{IP}_{0,1}$.

\medskip

{\bf 2.} Consider equivalent problems $\mathcal L = \mathcal L_{0,1}(\Sigma, U, V)$, $\tilde L = \mathcal L_{1,1}(\tilde \Sigma, \tilde U, \tilde V)$, where $\sigma_0 = -\tilde \sigma_0' \in L_1(0,1)$, $\sigma_2 = \tilde \sigma_2 \in L_2(0,1)$. If $\tilde l_{4,1}$ is fixed, then the one-to-one correspondence 
$$
(\sigma_0, \sigma_2, l_{2,1}, l_{3,1}, l_{3,2}, l_{4,1}, l_{4,2}, l_{4,3}) \leftrightarrow (\tilde \sigma_0, \tilde \sigma_2, \tilde l_{2,1}, \tilde l_{3,1}, \tilde l_{3,2}, \tilde l_{4,2}, \tilde l_{4,3})
$$
is given by \eqref{case2} and
$$
    \tilde \sigma_0(x) = \tilde l_{4,1} - l_{4,1} - \int_0^x \sigma_0(t) \, dt.
$$
Hence, $\mbox{IP}_{0, 1}$ is equivalent to $\mbox{IP}_{1,1}$.

\medskip

{\bf 3.} Consider equivalent problems $\mathcal L = \mathcal L_{1,0}(\Sigma, U, V)$, $\tilde L = \mathcal L_{2,0}(\tilde \Sigma, \tilde U, \tilde V)$, where $\sigma_0 = -\tilde \sigma_0' \in L_1(0, 1)$, $\sigma_0$ is continuous at zero, $\sigma_2 = \tilde \sigma_2 \in L_1(\mathbb R_+)$. If $\tilde l_{3,1}$ is fixed, then the one-to-one correspondence 
$$
(\sigma_0, \sigma_2, l_{2,1}, l_{3,1}, l_{3,2}, l_{4,1}, l_{4,2}, l_{4,3}) \leftrightarrow (\tilde \sigma_0, \tilde \sigma_2, \tilde l_{2,1}, \tilde l_{3,2}, \tilde l_{4,1}, \tilde l_{4,2}, \tilde l_{4,3})
$$
is given by \eqref{case3} and
$$
    \tilde \sigma_0(x) = l_{3,1} - \tilde l_{3,1} - \int_0^x \sigma_0(t) \, dt.
$$
Hence, $\mbox{IP}_{1, 0}$ is equivalent to $\mbox{IP}_{2,0}$.

\medskip

{\bf Funding.} This work was supported by Grant 21-71-10001 of the Russian Science Foundation, https://rscf.ru/en/project/21-71-10001/.

\medskip

{\bf Conflict of interest.} The author declares that this paper has no conflict of interest.

\noindent Natalia Pavlovna Bondarenko \\
1. Department of Applied Mathematics and Physics, Samara National Research University, \\
Moskovskoye Shosse 34, Samara 443086, Russia, \\
2. Department of Mechanics and Mathematics, Saratov State University, \\
Astrakhanskaya 83, Saratov 410012, Russia, \\
e-mail: {\it BondarenkoNP@info.sgu.ru}

\end{document}